\documentclass[12pt,a4paper]{article}

\usepackage{theorem,enumerate}
\usepackage{amsmath,latexsym,amssymb,amsfonts, epsfig}
\usepackage{eucal}
\usepackage{mathrsfs}
\usepackage{array}
\usepackage{epstopdf} 
\usepackage{MnSymbol}

\newcounter{Scounter}
\theorembodyfont{\normalfont\slshape}
\newtheorem{thm}{Theorem}[section]

\newtheorem{cor}[thm]{Corollary}
\newtheorem{prop}[thm]{Proposition}
\newtheorem{definition}[thm]{Definition}
\newtheorem{lemma}[thm]{Lemma}
 
\newtheorem{ob}[thm]{Observation}

\newtheorem{con}[thm]{Conjecture}
\newtheorem{prob}{Problem}

\newcommand{\qed}{{$\quad\square$\vs{3.6}}}

\newcommand{\vs}[1]{\vspace*{#1 mm}}

\bfseries\normalfont

\makeatletter
\def\thanks#1{%
   \footnotemark
   \edef\@tempa{\noexpand\noexpand\noexpand\footnotetext[\the\c@footnote]}%
   \toks@\expandafter{\@thanks}%
   \toks\tw@{{#1}}
   \xdef\@thanks{\the\toks@\@tempa\the\toks\tw@}}
\makeatother

\begin{document}

\title{Snarks with special spanning trees}


\author{
Arthur Hoffmann-Ostenhof, Thomas Jatschka}

\date{}
\maketitle


\begin{abstract} 
Let $G$ be a cubic graph which has a decomposition into a spanning tree $T$ and a $2$-regular subgraph $C$, i.e. $E(T) \cup E(C) = E(G)$ and $E(T) \cap E(C) = \emptyset$. We provide an answer to the following question: which lengths can the cycles of $C$ have if $G$ is a snark? Note that $T$ is a hist (i.e. a spanning tree without a vertex of degree two) and that every cubic graph with a hist has the above decomposition. 
\end{abstract}

\noindent
{\bf Keywords:} 
cubic graph, snark, spanning tree, hist, 3-edge coloring.

\section{Introduction}

For terminology not defined here, we refer to \cite{BM}. All considered graphs are finite and without loops.
A \textit{cycle} is a $2$-regular connected graph. A \textit{snark} is a cyclically $4$-edge connected cubic graph of girth at least $5$ admitting no $3$-edge coloring. Snarks play a central role for several well known conjectures related to flows and cycle covers in graph theory, see \cite{Z1}. \\
A \textit{hist} in a graph is a spanning tree without a vertex of degree two (hist is an abbreviation for homeomorphically irreducible spanning tree, see \cite{ABHT}). A \textit{decomposition} of a graph $H$ is a set of edge disjoint subgraphs covering $E(H)$. Note that a cubic graph $G$ has a hist if and only if $G$ has a decomposition into a spanning tree and a $2$-regular subgraph. 
For an example of a cubic graph with a hist, respectively, with the above decomposition, see for instance 
Figure \ref{f:petersen} or Figure \ref{f:t888} where the dashed edges illustrate the $2$-regular subgraph.\\ 
In general connected cubic graphs need not have a hist. 
Even cubic graphs with arbitrarily high cyclic edge-connectivity do not necessarily have a hist, see \cite{HO}. 
We call a snark $G$ a \textit{hist-snark} if $G$ has a hist. At first glance hist-snarks may seem very special.
However, using a computer and \cite{BCGM} (see also \cite{BGHM}), we recognize the following,

\begin{thm}\label{s}
Every snark with less than $38$ vertices is a hist-snark. 
\end{thm}

Observe that not every snark has a hist. There are at least two snarks with $38$ vertices which do not have a hist, see $X_1$ and $X_2$ in Appendix A3. \\
The following definition is essential for the entire paper. 

\begin{definition}\label{d:outer}
Let $G$ be a cubic graph with a hist $T$. \\
{\bf(i)} An outer cycle of $G$ is a cycle of $G-E(T)$. \\
{\bf(ii)} Let $\{C_1,C_2,...,C_k\}$ be the set of all outer cycles of $G$ with respect to $T$, then we denote by $oc(G,T)=\{|V(C_1)|,|V(C_2)|,...,|V(C_k)|\}$.
\end{definition}

Note that all vertices of the outer cycles in Definition \ref{d:outer} are leaves of the hist. Observe also that $oc(G,T)$ is a multiset since several elements of $oc(G,T)$ are possibly the same number, see for instance the hist-snark in Figure \ref{f:t888}. 
If we refer to the outer cycles of a hist-snark, then we assume that the hist of the snark is given and thus the outer cycles are well defined (a hist-snark may have several hists). This paper answers the following type of problem.
For any $m \in \mathbb N$, is there a hist-snark with precisely one outer cycle such that additionally the outer cycle has length $m$? Corollary \ref{c:one} answers this question and the more general 
problem is solved by the main result of the paper: 


\begin{thm}\label{t:main}
Let $S=\{c_1,c_2,...,c_k\}$ be a multiset of $k$ natural numbers. Then there is a snark $G$ with a hist $T$ such that $oc(G,T)=S$ if and only if the following holds: \\(i) $c_1=6$ or $c_1 \geq 10$, if $k=1$. \\
(ii) $c_j \geq 5$ for $j=1,2,...,k$ if $k > 1$.
\end{thm}

One of our motivations to study hist-snarks is a conjecture on cycle double covers, 
see Conjecture \ref{c:cover}. Note that this conjecture has recently be shown to hold for certain classes of hist-snarks, see \cite{HZZ}. Special hist-snarks with symmetric properties can be found in \cite{HJ}. For examples of hist-snarks within this paper, see Figures \ref{f:loup}, \ref{f:petersen}, \ref{f:blan10}, \ref{f:t888} (dashed edges illustrate outer cycles). For a conjecture which is similar to the claim that every connected cubic graph $G$ has a decomposition into a spanning tree and a $2$-regular subgraph, see the $3$-Decomposition Conjecture in \cite{HKO}.

\section{On the lengths of outer cycles of hist-snarks}

To prove Theorem \ref{t:main}, we develop methods to construct hist-snarks. Thereby we use in particular modifications of the dot product
 and a handful of computer generated snarks. In all drawings of this section, dotted thin edges symbolize removed edges whereas dashed edges illustrate edges of outer cycles. The graphs which we consider may contain multiple edges.\\
The \textit{neighborhood} $N(v)$ of a vertex $v$ denotes the set of vertices adjacent to $v$ and does not include $v$ itself. The cyclic edge-connectivity of a graph $G$ is denoted by $\lambda_c(G)$. \textit{Subdividing} an edge $e$ means to replace $e$ by a path of length two. \\
We define the dot product (see Fig.\ref{f:dot}) which is a known method to construct snarks, see \cite{AT,I}. Let $G$ and $H$ be two cubic graphs. Let $e_1=a_1b_1$ and $e_2=a_2b_2$ be two independent edges of $G$ and let $e_3$ with $e_3=a_3b_3$ be an edge of $H$ with $N(a_3)-b_3=\{x_1,y_1\}$ and $N(b_3)-a_3=\{x_2,y_2\}$. The \textit {dot product} $G \cdot H$ is the cubic graph $$(G \cup H -a_3 - b_3 - \{e_1,e_2\}) \cup \{a_1x_1,b_1y_1,a_2x_2,b_2y_2\} \,\,.$$ 

Both results of the next lemma are well known. For a proof of Lemma \ref{l:3color}(i), see for instance 
\cite[p.~69]{Z1}. Lemma \ref{l:3color}(ii) is folklore (a published proof is available in \cite{AT}).
 
\begin{lemma}\label{l:3color}
Let $G$ and $H$ be both $2$-edge connected cubic graphs, then the following holds\\ 
(i) $G \cdot H$ is not $3$-edge colorable if both $G$ and $H$ are not $3$-edge colorable.\\
(ii) $G \cdot H$ is cyclically $4$-edge connected if both $G$ and $H$ are 
cyclically $4$-edge connected.
\end{lemma}


We use a modification of a dot product where $\{h_1,j_1,h_2,j_2\}$ is a set of four distinct vertices which is disjoint with $V(G) \cup V(H)$:

\begin{definition}\label{d:var}
Set $G(\underline{e_1},e_2) \bullet H(e_3):=\big(G \cdot H -\{a_1x_1,b_1y_1\}\big) \cup \{h_1,j_1,a_1h_1,h_1j_1,j_1b_1,h_1x_1,j_1y_1\} $, 
$G(e_1,\underline{e_2}) \bullet H(e_3):= \big(G \cdot H - \{a_2x_2,b_2y_2\}\big) \cup \{h_2,j_2,a_2h_2,h_2j_2,j_2b_2,h_2x_2,j_2y_2 \}$ 
and \\ $G(\underline{e_1},\underline{e_2}) \bullet H(e_3):= 
\big(G(\underline{e_1},e_2) \bullet H(e_3) -\{a_2x_2,b_2y_2\}\big) 
\cup \{h_2,j_2,a_2h_2,h_2j_2,j_2b_2,h_2x_2,j_2y_2\}$, see Fig.\ref{f:dot}.
 \end{definition}

\begin{figure}[htpb] 
\centering\epsfig{file=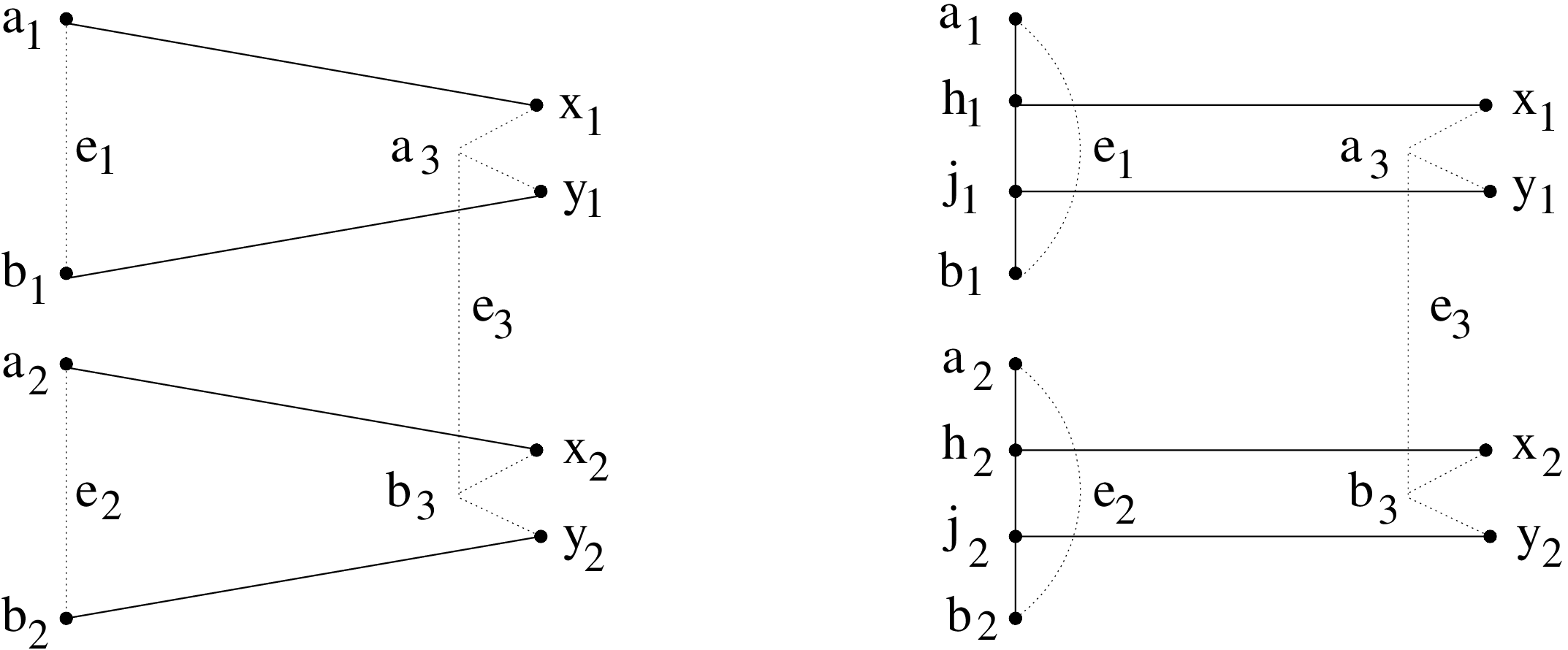,width=4.3in}
\caption{Constructing the dot product $G \cdot H$ and $G(\underline{e_1},\underline{e_2}) \bullet H(e_3)$, see Def. \ref{d:var}.}
\label{f:dot}
\end{figure}

The statement of the next proposition is well known, see \cite{W} or \cite[proof of Theorem 12]{A} (see also \cite{F1,F2}).

\begin{prop}\label{p:444} 
Let $H$ be a cubic graph which is constructed from a cubic graph $G$ by subdividing two independent edges in $G$
and by adding an edge joining the two $2$-valent vertices. Then $\lambda_c(H) \geq 4$ if $\lambda_c(G) \geq 4$.
\end{prop} 

The above result is used for the proofs of the subsequent two lemmas. 

\begin{lemma}\label{l:dot}
Suppose $G$ and $H$ are snarks. Define the three graphs: $B_1:=G(\underline{e_1},e_2) \bullet H(e_3)$,  $B_2:=G(e_1,\underline{e_2}) \bullet H(e_3)$ and $B_3=:G(\underline{e_1},\underline{e_2}) \bullet H(e_3)$. Then $B_i$ is a snark for $i=1,2,3$. 
\end{lemma}

\begin{proof}
If $B_i$ with $i=1,2,3$ has a cycle of length less than five, then this cycle must contain precisely two edges of the cyclic $4$-edge cut $C_i$ of $B_i$ where $C_i$ is defined by the property that one component of $B_i-C_i$ is $H-a_3-b_3$. 
Since the endvertices of the edges of $C_i$ do not induce a $4$-cycle, $B_i$ has girth at least $5$. \\
It is not difficult to see that $B_i$ results from subdividing two (in the case $i=1,2$) or four (in the case $i=3$) independent edges of $G \cdot H$ (namely the edges of $C_i$) and adding one or two new edges. 
Since $\lambda_c(G \cdot H) \geq 4$ by Lemma \ref{l:3color}(ii) and since subdividing and adding edges as described keeps by Proposition \ref{p:444} the cyclic $4$-edge connectivity, $\lambda_c(B_i)\geq 4$.\\
It remains to show that $B_i$ is not $3$-edge colorable. 
Let $G_1$ be the cubic graph which is constructed from $G$ by (i) replacing $e_1 \in E(G)$ by a path of length three (no edge in our graph is now labeled $e_1$),
by (ii) removing the labels $a_1$, $b_1$, by (iii) adding a parallel edge (to make the graph cubic) and calling this new edge $e_1$, and by (iv) naming the 
endvertices of $e_1$, $a_1$ and $b_1$ such that $a_1$ is adjacent to the vertex whose label $a_1$ we removed in step (ii).   
Then $B_1 \cong G_1 \cdot H$. Since $G_1$ is by construction clearly not $3$-edge colorable and since $H$ is a snark, Lemma \ref{l:3color}(i) implies that $B_1$ is not $3$-edge colorable. The remaining cases $B_2$ and $B_3$ can be verified analogously. \qed
\end{proof}

Let $e_1,e_2$ in $G$ and $e_3$ in $H$ be defined as at the beginning of this section. 
Moreover, suppose that $N(b_1) =\{a_1,c,d\}$ and that all three neighbors of $b_1$ are distinct. We define another modification of the dot product, 
see also Fig.\ref{f:triangel}.

\begin{definition}\label{d:triangel}
Let $q_1$,$q_2$ be distinct vertices satisfying $\{q_1,q_2\} \cap V(G \cup H) = \emptyset$. Set \\
$G(e_1,e_2) \blacktriangle H(e_3):= \big(G \cup H -a_3 - b_3 - \{e_1,e_2,b_1c\}\big) \cup 
\{q_1,q_2,a_1q_1,q_1b_1,b_1q_2,q_2c,q_1x_1,q_2y_1,a_2x_2,b_2y_2\}$.
\end{definition}

\begin{figure}[htpb] 
\centering\epsfig{file=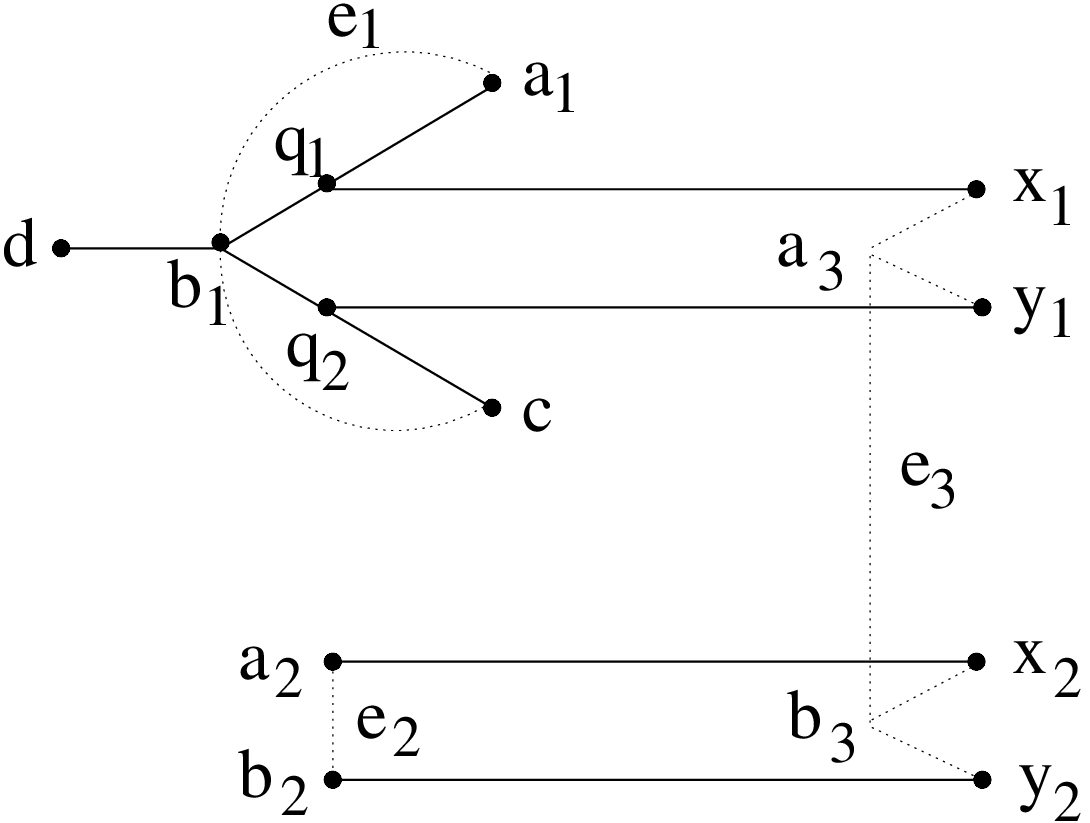,width=2.4in}
\caption{Constructing the graph $G(e_1,e_2) \blacktriangle H(e_3)$, see Def. \ref{d:triangel}.}
\label{f:triangel}
\end{figure}

\begin{lemma}\label{l:triangel}
Suppose $G$ and $H$ are snarks, then $G(e_1,e_2) \blacktriangle H(e_3)$ is a snark.
\end{lemma}

\begin{proof}
By the same arguments as in the proof of Lemma \ref{l:dot}, the girth of $G(e_1,e_2) \blacktriangle H(e_3)$ is at least five. Set $X:=G \cdot H$. 
Subdivide in $X$ the edge $a_1x_1$ and call the obtained $2$-valent vertex
$q_1$, then subdivide the edge $b_1d$ and call this now obtained $2$-valent vertex $q_2$, then 
exchange the labels $b_1$ and $q_2$ and finally add $b_1q_1$. 
Thus we obtain the graph $Y:=G(e_1,e_2) \blacktriangle H(e_3)$. Since $\lambda_c(X) \geq 4$ by Lemma \ref{l:3color}(ii), and since $Y$ is obtained from $X$ by subdividing two independent edges and by adding an edge joining these vertices, 
it follows from Proposition \ref{p:444} that $\lambda_c(Y) \geq 4$. It remains to show that $Y$ is not $3$-edge colorable. 
Let $\tilde G$ be the cubic graph which
is obtained from $G$ by (i) expanding $b_1$ to a triangle, by (ii)  
removing the labels $e_1$, $a_1$ and $b_1$, by (iii) calling the edge of the triangle which has no endvertex adjacent to $d$, $e_1$, and by (iv) naming the endvertices of $e_1$, $a_1$ and $b_1$ such that $a_1$ is adjacent to the vertex whose label $a_1$ we removed in step (ii). 
Then $\tilde G$ is clearly not $3$-edge colorable. By Lemma \ref{l:3color}(i) and since $Y \cong \tilde G \cdot H$, $Y$ is not $3$-edge colorable. \qed
\end{proof}

The next result shows that two hist-snarks can be combined to generate new hist-snarks.

\begin{thm}\label{t:union}
Let $G$ and $H$ be snarks with hists $T_G$ and $T_H$, then the following holds:\\
(i) There is a snark $G'$ with a hist $T'$
such that $oc(G',T')=oc(G,T_G) \cup \,oc(H,T_H)$. \\
(ii) Suppose $k \in oc(G,T_G)$ and $l \in oc(H,T_H)$, then there is a snark $\hat G$ with a hist 
$\hat T $ such that $oc(\hat G, \hat T)= \big(oc(G, T_G) \cup \,oc(H,T_H) \cup \{k+l-1\}\big) - \{k,l\}$.
\end{thm}

\begin{proof}
First we prove (i).
Choose two edges $e_1,e_2 \in E(T_G)$ and choose an edge $e_3 \in E(T_H)$ such that all four adjacent edges are part of $T_H$. Set $G':=G(\underline{e_1},\underline{e_2}) \bullet H(e_3)$. Then $G'$ is a snark by Lemma \ref{l:dot}.
Let $T'$ be the subgraph of $G'$ with $E(T')=(E(T_G) \cup E(T_H)-\{e_1,e_2,e_3,a_3x_1,a_3y_1,b_3x_2,b_3y_2\}) \cup \{a_1h_1,h_1j_1,j_1b_1,a_2h_2,h_2j_2,j_2b_2,h_1x_1,j_1y_1,h_2x_2,j_2y_2\}$. \\ Since $T_H-a_3-b_3$ consists of four components, it follows that $T'$ is acyclic. It is straightforward to verify that $T'$ is a hist of $G'$. Since every outer cycle of $G'$ is an outer cycle of $G$ or $H$, the proof of (i) is finished. 

Using Lemma \ref{l:triangel}, 
we define a snark $\hat G=G(e_1,e_2) \blacktriangle H(e_3)$ where $e_1,e_2,e_3$ are chosen to satisfy the following properties. In $G$, let $e_1 \in E(T_G)$, $b_1c \in E(T_G)$ and let $e_2 \in E(G)-E(T_G)$ be part 
of an outer cycle of length $k$, see Fig.\ref{f:triangel}. In $H$, let $b_3$ be a leaf of an outer cycle of length $l$ and let $e_3,a_3x_1,a_3y_1 \in E(T_H)$. Let $\hat T$ be the the subgraph of $\hat G$ with 
$E(\hat T):= (E(T_G) \cup E(T_H)-\{e_1,b_1c,e_3,a_3x_1,a_3y_1\}) \cup \{a_1q_1,q_1b_1,b_1q_2,q_2c,q_1x_1,q_2y_1\}$. It is straightforward to verify that $\hat T$ is a hist of $\hat G$. Note that $a_2x_2,b_2y_2$ are contained in an outer cycle of length $k+l-1$. Hence, $oc(\hat G, \hat T)= \big(oc(G, T_G) \cup \,oc(H,T_H)-\{k\}-\{l\}\big) \cup \{k+l-1\}$ which finishes the proof. \qed
\end{proof}

\begin{lemma}\label{l:reduction}
Let $G$ be a snark with a hist $T_G$ and let $k \in oc(G,T_G)$, then each of the following statements holds:\\
(i) there is a snark $G'$ with a hist $T'$ such that $oc(G',T')=\big(oc(G,T_G)-\{k\}\big) \cup \{k+4\}$.\\
(ii) there is a snark $G'$ with a hist $T'$ such that $oc(G',T')=oc(G,T_G) \cup \{5\}$.\\
(iii) there is a snark $G'$ with a hist $T'$ such that $oc(G',T')=oc(G,T_G) \cup \{6\}$.\\
(iv) there is a snark $G'$ with a hist $T'$ such that $oc(G',T')=\big(oc(G,T_G)-\{k\}\big) \cup \{k+2\} \cup \{7\}$.
\end{lemma}

\begin{proof}
The endvertices of the edges $e_1$, $e_2$ and $e_3$ and their neighbors are labeled as defined in the beginning of this section. The Petersen graph is denoted by $P_{10}$.

{\bf (i)} 
Let $\hat C$ be the $5$-cycle, "the inner star" of the illustrated $P_{10}$ in 
Fig.\ref{f:4}. Set $U:=P_{10}-E(\hat C)$ and let $e_3 \in E(U)$ with $e_3=a_3b_3$ where 
$b_3 \in V(\hat C)$. In $G$, we choose two independent edges $e_1 \in E(T_G)$ and $e_2 \not\in E(T_G)$ where $e_2$ is part of an outer cycle $C_k$ of length $k$. Then $G':=G \cdot P_{10}$ is a snark and the subgraph $T'$ of $G'$ with 
$E(T'):=(E(T_G)-e_1) \cup \{a_1x_1,b_1y_1\} \cup E(U-a_3)$ is a hist of $G'$, see Fig.\ref{f:4}. Since $G$ and $G'$ have the same outer cycles with the only exception that $C_k$ 
is in $G$ and that $C'_{k+4}:=(C_k-e_2) \cup \{a_2x_2,b_2y_2\} \cup (\hat C-b_3)$ is in $G'$, the statement follows.

\begin{figure}[htpb] 
\centering\epsfig{file=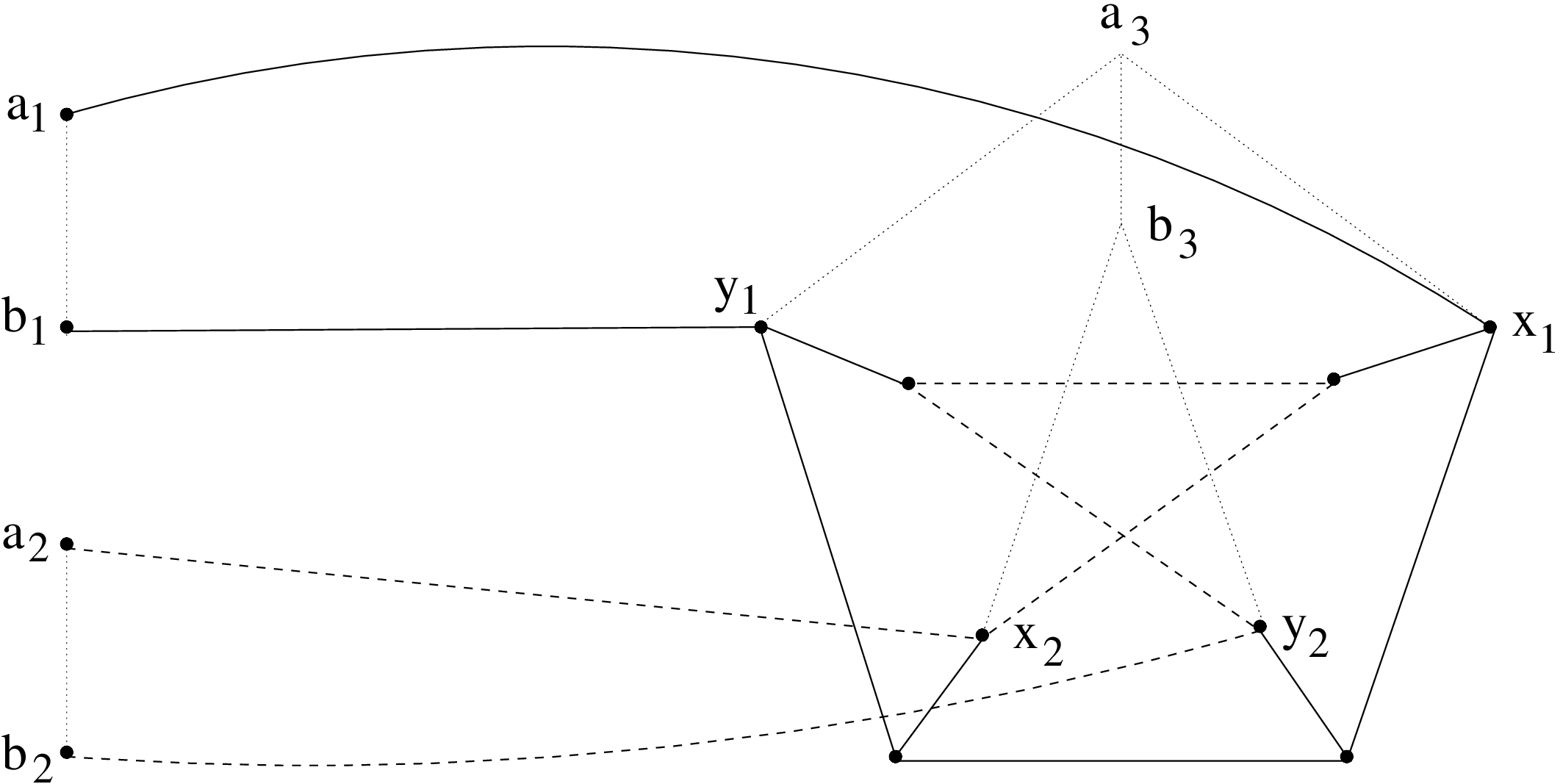,width=3.9in}
\caption{Constructing the graph $G'$ in the proof of Lemma \ref{l:reduction} (i).}
\label{f:4}
\end{figure}

{\bf (ii)}   
Let $e_1,e_2 \in E(T_G)$ with $e_1=a_1b_1$ and $e_2=a_2b_2$. Since we could exchange the vertex labels $a_2$ and $b_2$, we can assume that $b_1$, $b_2$ are in the same component of $T_G-e_2$. 
Define the snark $G':=G(\underline{e_1},e_2) \bullet P_{10}(e_3)$ with $e_3 \in E(P_{10})$, see Fig.\ref{f:5}. Let $C_5$ and $\hat C_5$ be two disjoint $5$-cycles of $P_{10}$ with $e_3 \in E(C_5)$ and let $x_1,x_2 \in V(C_5)$, see Fig.\ref{f:5}. Then the subgraph $T'$ of $G'$ with $E(T'):=(E(T_G)-\{e_1,e_2\}) \cup \{a_1h_1,h_1j_1,j_1b_1,h_1x_1,j_1y_1,a_2x_2,b_2y_2\} \cup 
E(P_{10})-(E(\hat C_5) \cup \{a_3x_1,a_3y_1,b_3x_2,b_3y_2,e_3\})$ is a hist of $G'$, see Fig.\ref{f:5}. 
Since the set of outer cycles of $G'$ consists of $\hat C_5$ and all outer cycles of $G$, the statement follows. \\

\begin{figure}[htpb] 
\centering\epsfig{file=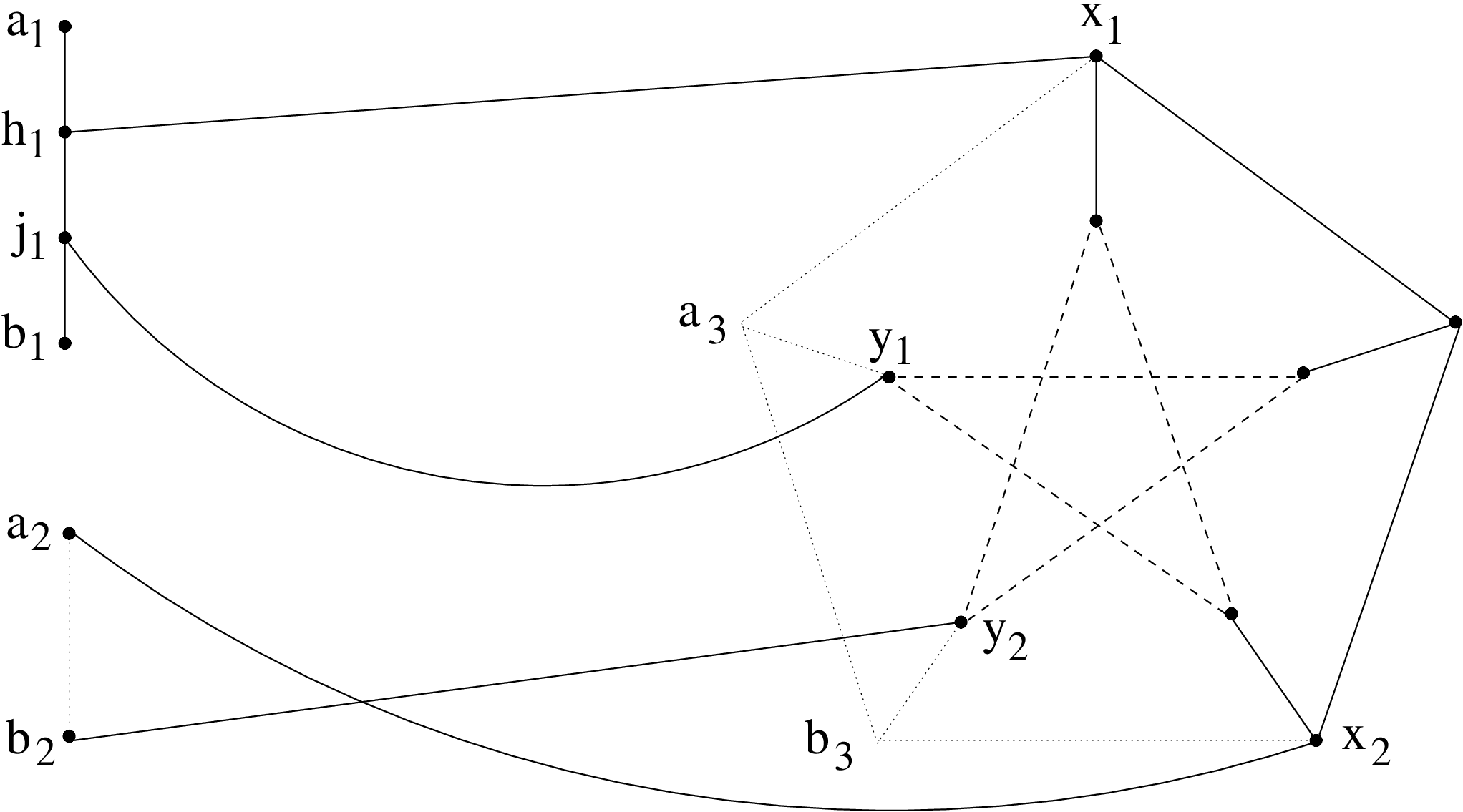,width=3.8in}
\caption{Constructing the graph $G'$ in the proof of Lemma \ref{l:reduction} (ii).}
\label{f:5}
\end{figure}

{\bf (iii)}
$P_{10}$ is a hist-snark with one outer cycle, see Fig.\ref{f:petersen}. Since this cycle has length $6$, statement (iii) follows from applying Theorem \ref{t:union}(1) by setting $H:=P_{10}$. 

{\bf (iv)} Let $e_1=a_1b_1$ be an edge of an outer cycle of length $k$ in $G$ and let $e_2 \in E(T_G)$. 
Let $B_{18}$ denote the Blanusa snark, see Fig.\ref{f:blan10}. Define the snark $G':=G(\underline{e_1},e_2) \bullet B_{18}(e_3)$ as illustrated in 
Fig.\ref{f:blan}. 
Note that two outer cycles of lengths $7$ and $k+2$ are presented in Fig.\ref{f:blan} by dashed lines. Let $B$ denote the edge set 
which contains all edges of $G'$ which are shown in bold face in Fig.\ref{f:blan}. It is straightforward to verify that 
the subgraph $T'$ of $G'$ with $E(T'):= (E(T_G)-\{e_2\}) \cup B$ is a hist in $G'$ satisfying $oc(G',T')=(oc(G,T_G)-\{k\}) \cup \{k+2\} \cup \{7\}$. \qed
\end{proof}

\begin{figure}[htpb] 
\centering\epsfig{file=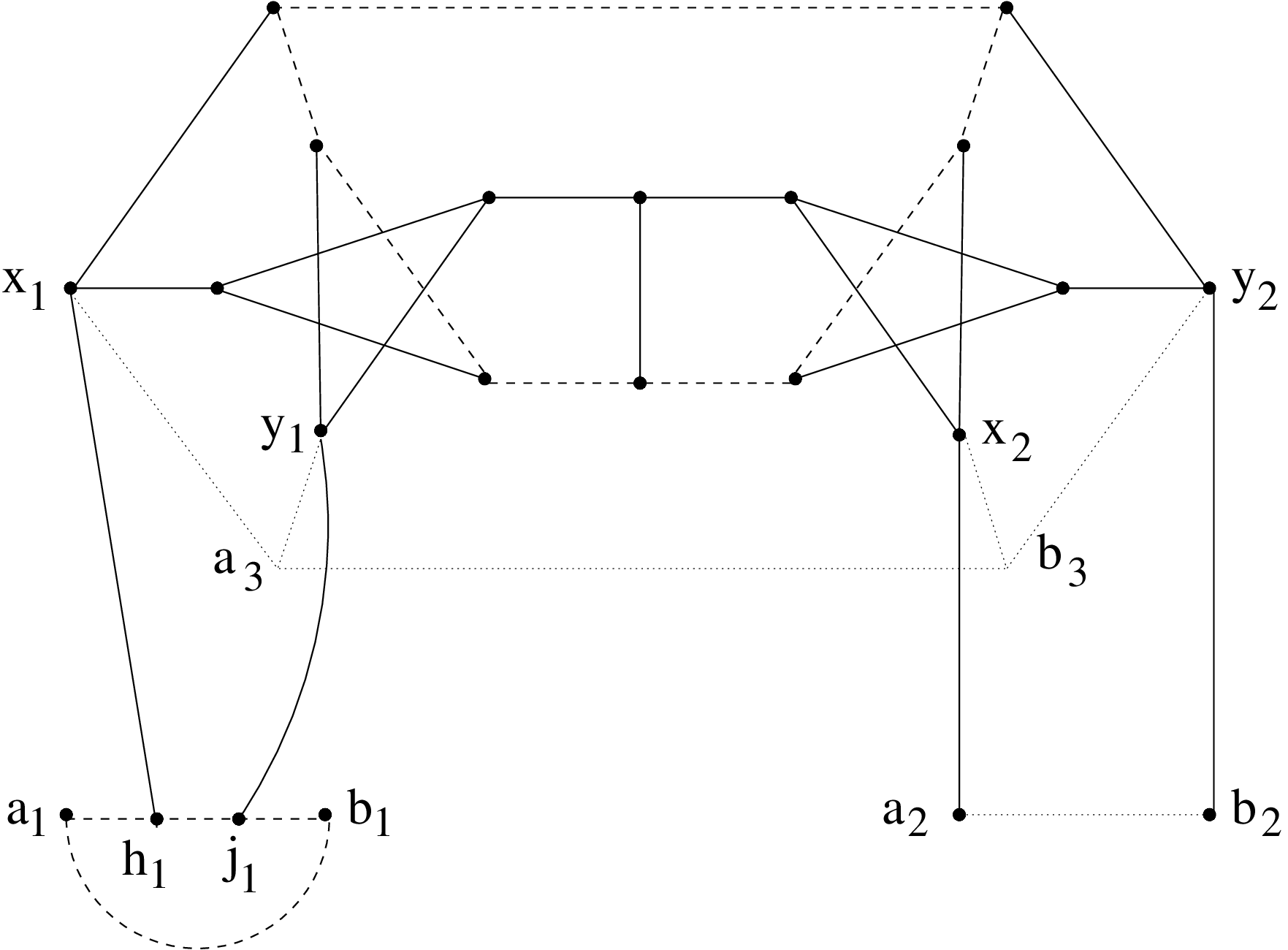,width=3.8in}
\caption{Constructing the graph $G'$ in the proof of Lemma \ref{l:reduction} (iv).}
\label{f:blan}
\end{figure}

\begin{definition}\label{d:S}
Let $S$ be a multiset of natural numbers, then 
$S^*$ denotes the set of all hist-snarks $G$ which have a hist $T_G$ such that $oc(G,T_G)=S$.
\end{definition}

For instance, the Blanusa snark $B_{18}$ satisfies $B_{18} \in \{10\}^*$ (see Fig.\ref{f:blan10}) 
but also satisfies $B_{18} \in \{5,5\}^*$. We leave it to the reader to verify the latter fact. 
The Petersen graph $P_{10}$ satisfies $P_{10} \in S^*$ if and only if $S=\{6\}$, 
see Fig.\ref{f:petersen}.


\begin{cor}\label{c:one}
There is a snark $G$ having a hist $T$ with $oc(G,T)=\{k\}$ if and only if $k=6$ or $k \geq 10$.
\end{cor}

\begin{proof}
As mentioned above $B_{18} \in \{10\}^*$, see Fig.\ref{f:blan10}.
By Theorem \ref{t:union} (2) and since $P_{10} \in \{6\}^*$, it follows that $\{11\}^* \not=\emptyset$. 
The second Loupekine snark denoted by $L_{22}$ satisfies $L_{22} \in \{12\}^*$, see Fig.\ref{f:loup}. The hist-snark $T(13)$, see Appendix A1 satisfies $T(13) \in \{13\}^*$. 
Applying Lemma \ref{l:reduction} (i) to each of these four hist-snarks and proceeding inductively, we
obtain for every natural number $k \geq 10$ a hist-snark $G$ satisfying $G \in \{k\}^*$. 
Since $k=|V(G)|/2+1$ (see Theorem 2 in \cite{HO}) and since there is no snark with $10<k<18$ vertices, $\{l\}^* =\emptyset$ for every $l \in \{1,2,3,4,5,7,8,9\}$.
Finally, since $P_{10} \in \{6\}^*$ the proof is finished. \qed
\end{proof}

\begin{figure}[htpb] 
\centering\epsfig{file=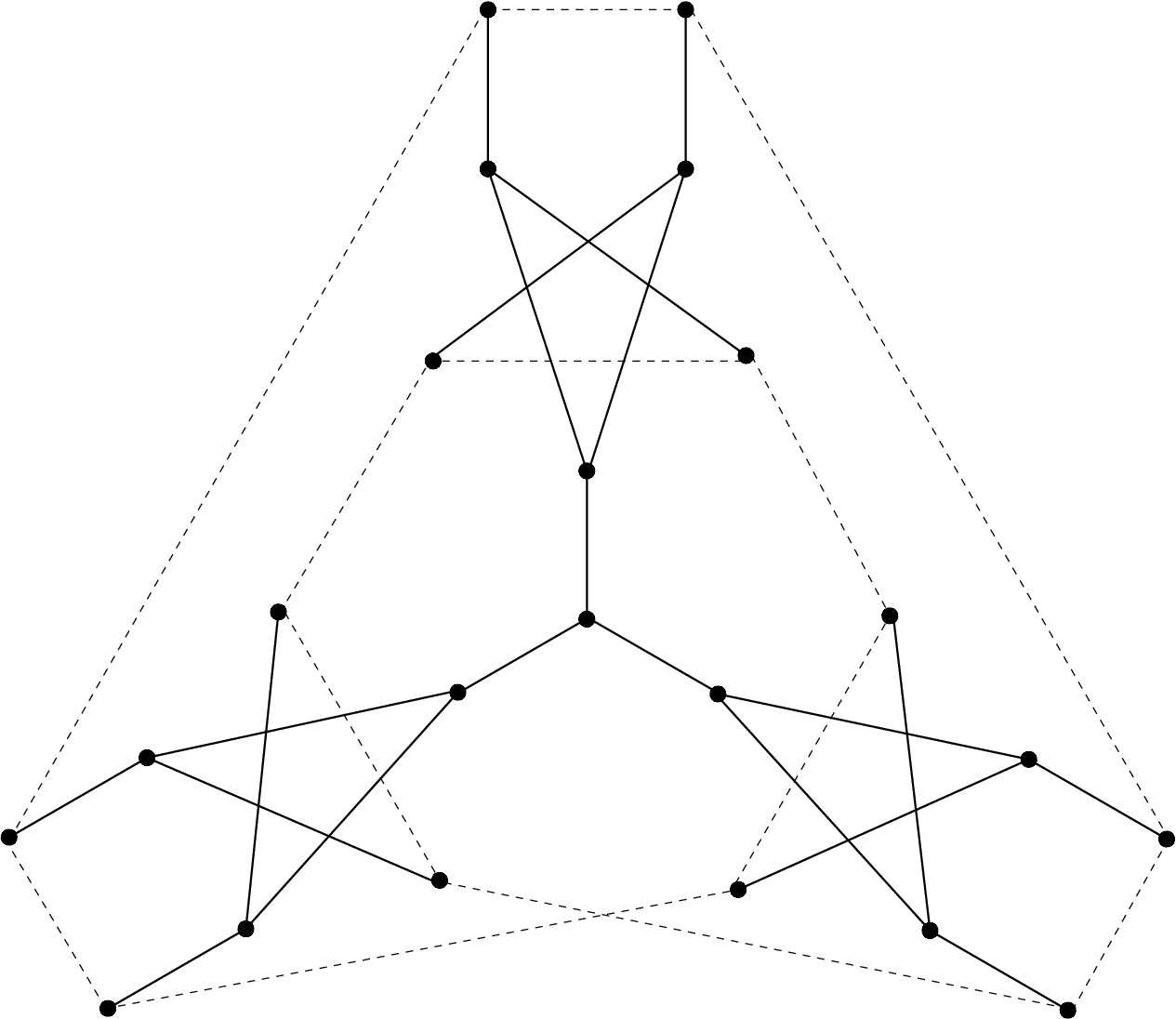,width=3in}
\caption{The second Loupekine snark with an outer cycle of length $12$.}
\label{f:loup}
\end{figure}


\begin{lemma}\label{l:two}
Let $S=\{x,y\}$ with $x,y \in \{5,6,7,8\}$, then $S^* \not= \emptyset$.
\end{lemma}

\begin{proof}
By applying Lemma \ref{l:reduction} (ii),(iii),(iv) by setting $G:=P_{10}$, we obtain that 
$\{5,6\}^*\not= \emptyset$, $\{6,6\}^*\not= \emptyset$, $\{7,8\}^*\not= \emptyset$. To avoid more constructions, we used a computer. We refer to
Appendix A1 and Fig.\ref{f:liste}, where one member of $\{x,y\}^*$ denoted by $T(x,y)$ 
is presented for the remaining pairs.
\end{proof}

\begin{lemma}\label{l:three}
Let $S=\{x,y,z\}$ with $x,y,z \in \{5,6,7,8\}$, then $S^* \not= \emptyset$.
\end{lemma}

\begin{proof}
By Lemma \ref{l:reduction} (ii), Lemma \ref{l:reduction} (iii) and Lemma \ref{l:two}, $S^*\not= \emptyset$ if $\{5,6\} \cap S \not= \emptyset$.
Hence we assume that $x,y,z \in \{7,8\}$. 
Suppose $7 \in S$ and let without loss of generality $x=7$. 
By Lemma \ref{l:two}, $\{y-2,z\}^* \not= \emptyset$. Applying Lemma \ref{l:reduction} (iv), we obtain that $\{7,y,z\}^* \not= \emptyset$. 
Since there is a snark $T(8,8,8)$ (see Fig.\ref{f:t888} in Appendix A2) which is a member of $\{8,8,8\}^*$, the lemma follows. \qed
\end{proof}

Note that the illustrated snark in $\{8,8,8\}^*$ has a $2\pi/3$ rotation symmetry and a hist which has equal distance from its central root to every leaf. Such snarks 
are called \textit{rotation snarks}, for an exact definition see \cite{HJ}. The Petersen graph and both Loupekine's snarks (the smallest cyclically $5$-edge connected snarks apart from the 
Petersen graph) are rotation snarks. All rotation snarks with at most $46$ vertices are presented in \cite{HJ}.

\begin{figure}[htpb] 
\centering\epsfig{file=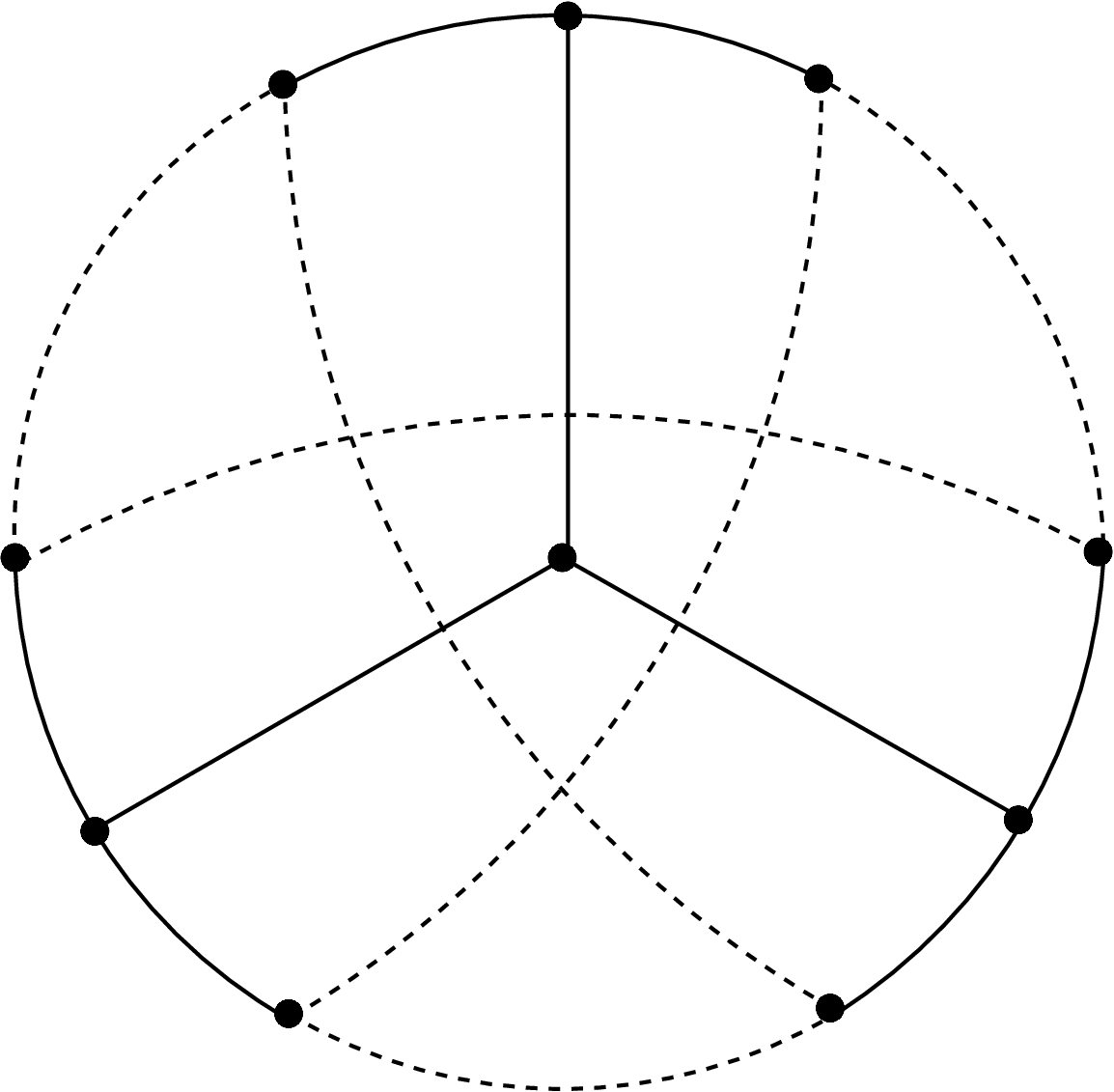,width=1.8in}
\caption{The Petersen graph with an outer cycle of length $6$.}
\label{f:petersen}
\end{figure}




\medbreak\noindent\textit{Proof of Theorem \ref{t:main}.}\quad
Statement (i) is implied by Corollary \ref{c:one}. 
Statement (ii) is obviously a necessary condition, otherwise $G$ has girth less than five. 
Hence it suffices to show that there is a hist-snark in $S^*$ if $S$ satisfies (ii).  \\
Suppose $S$ is a counterexample which is firstly minimal with respect to $|S|$ and secondly minimal with respect to the largest number in $S$. Suppose $|S| \geq 4$. Then there is a partition $S=S_2 \cup S_3$ with $|S_2|=2$ and $|S_3|=|S|-2 \geq 2$. Since the elements of $S_2$, $S_3$ satisfy (ii), there is by minimality a hist-snark $H_i \in S^*_i, i=2,3$. By Theorem \ref{t:union}, there is a hist-snark in $(S_2 \cup S_3)^*=S^*$ which is a contradiction. Hence $|S| \in \{2,3\}$. \\Suppose $m \in S$ and $m > 8$. Set $S_1=(S-\{m\}) \cup \{m-4\}$. Obviously the elements of $S_1$ 
fulfill (ii) and thus there is by minimality a hist-snark $H_1 \in S^*_1$. Applying Lemma 
\ref{l:reduction} (i) to $S_1$, we obtain a hist-snark in $S^*$ which is a contradiction.
Thus, $S$ consists of two or three elements and each of them is contained in $\{5,6,7,8\}$. By Lemma \ref{l:two} and 
Lemma \ref{l:three}, this is not possible. \qed

\begin{figure}[htpb] 
\centering\epsfig{file=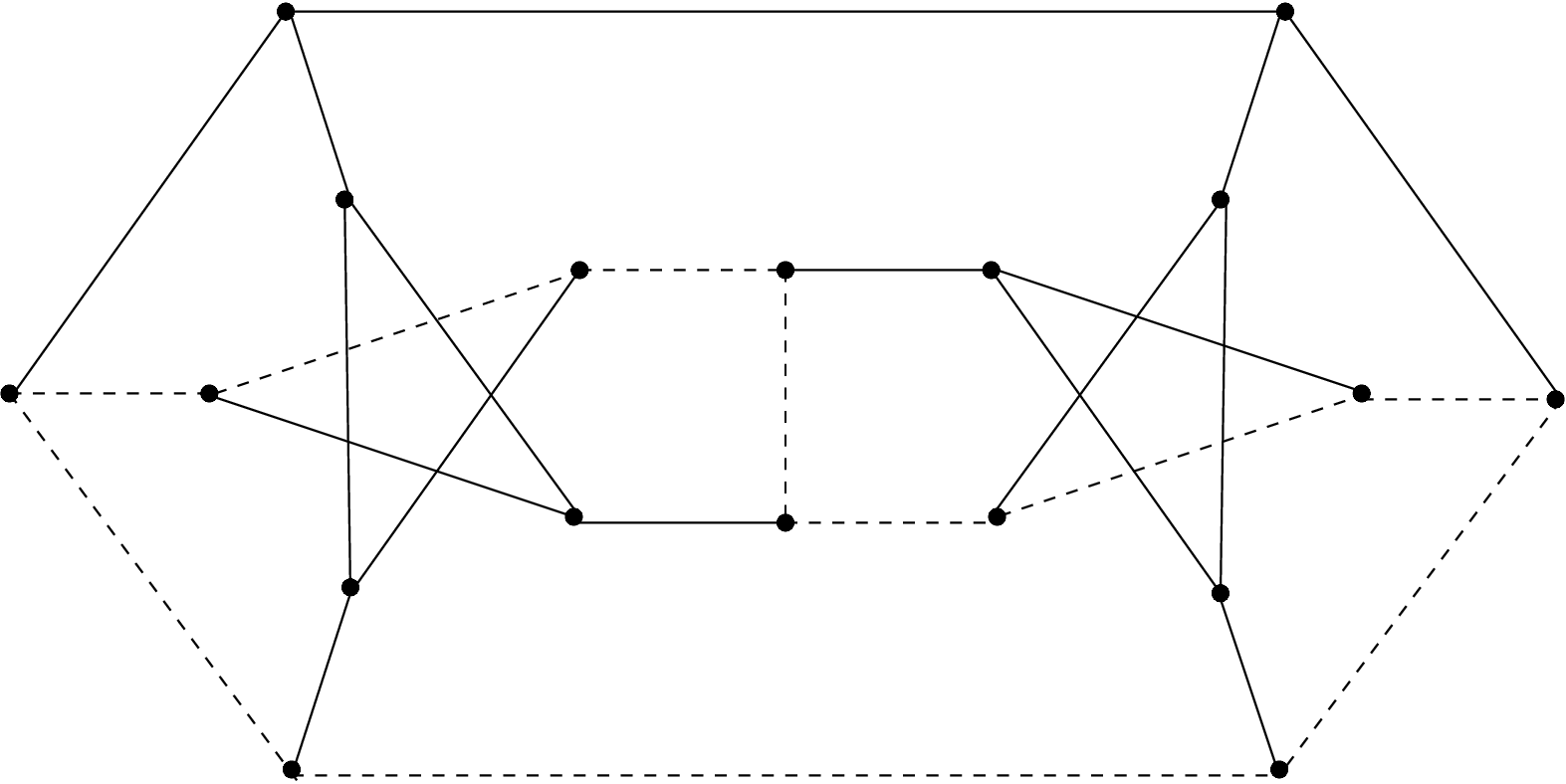,width=3.1in}
\caption{The Blanusa snark $B_{18}$ with an outer cycle of length $10$.}
\label{f:blan10}
\end{figure}

\section{Open problems}

The following conjecture by the first author, was presented firstly at the 9th Workshop on the Matthews-Sumner Conjecture and Related Problems in Pilsen in 2017. 

\begin{con}\label{c:cover}
Every hist-snark has a cycle double cover which contains all outer cycles.
\end{con}

The above conjecture is motivated by the following observation on hist-snarks.

\begin{ob}\label{c:cor}\cite{HO3}
Let $G$ be a snark with a hist $T$. Suppose there is a matching $M$ of $G$ satisfying $M \subseteq E(G)-E(T)$, and suppose the cubic graph homeomorphic to $G-M$ is $3$-edge colorable. Then $G$ has a cycle double cover containing all outer cycles of $G$.
\end{ob}

We omit here a proof of Observation \ref{c:cor} since Theorem 3.2 in \cite{HZZ} implies Observation \ref{c:cor}. 
Note that Conjecture \ref{c:cover} is already known to hold for all hist-snarks which have at most three outer cycles, see \cite{HZZ}. 
 

\section*{Acknowledgments}
A.Hoffmann-Ostenhof was supported by the Austrian Science Fund (FWF) project P 26686. The computational results presented have been achieved by using the Vienna Scientific Cluster (VSC).

\section{Appendix}

{\bf A1.} The following hist-snarks are defined by the corresponding hists illustrated below and the outer cycles whose vertices are 
presented within brackets in cyclic order. \\

\noindent
$T(5,5):= [10,15,14,17,16]\,\,[2,7,3,8,9]$\\
$T(5,7):= [3,15,13,17,16]\,\,[10,4,2,21,20,18,11]$\\
$T(5,8):= [3,15,13,17,16]\,\,[10,4,2,23,22,18,11,21]$\\
$T(6,7):= [1,6,7,19,18,22]\,\,[4,5,17,13,15,14,10]$\\
$T(6,8):= [18,19,14,21,20,23]\,\,[1,5,4,2,7,6,24,16]$\\
$T(7,7):= [17,13,15,14,10,25,24]\,\,[1,6,7,19,2,23,22]$\\
$T(8,8):= [12,9,8,29,28,4,5,13]\,\,[14,15,18,21,24,26,19,23]$\\
$T(13):= [2,19,7,3,15,13,17,5,20,21,11,9,23]$\\


\begin{figure}[htpb] 
\centering\epsfig{file=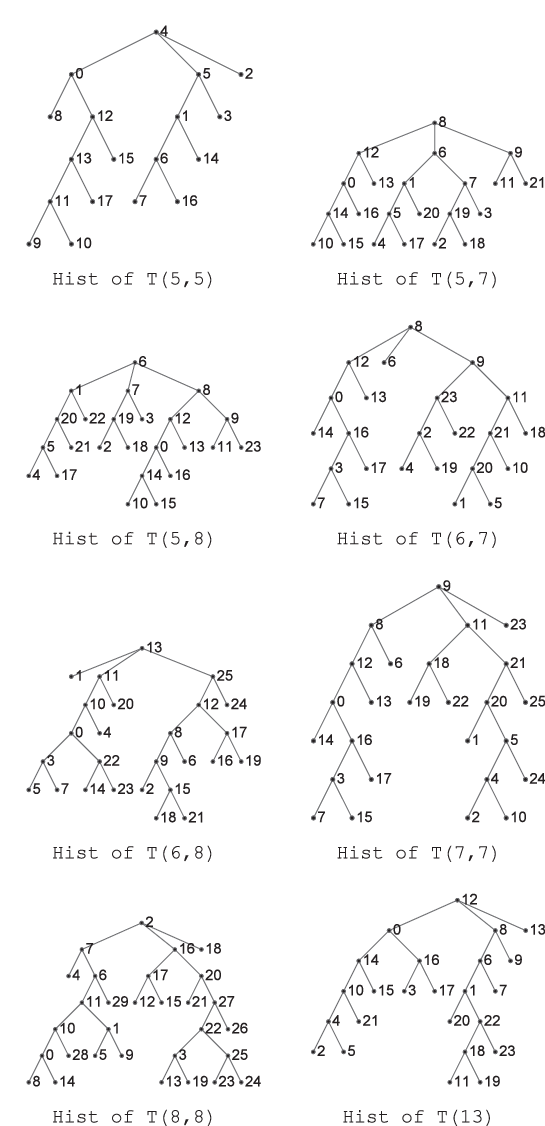,width=4.4in}
\caption{See Appendix A1, proof of Corollary \ref{c:one} and proof of Lemma \ref{l:two}.}
\label{f:liste}
\end{figure}


The adjacency lists of the above hist-snarks: \\

$T(5,5):0(4,8,12)1(5,6,14)2(4,7,9)3(5,7,8)4(5)6(7,16)8(9)9(11)10(11,15,16)11(13)\\12(13,15)13(17)14(15,17)16(17)$\\


$T(5,7):0(12,14,16)1(5,6,20)2(4,19,21)3(7,15,16)4(5,10)5(17)6(7,8)7(19)8(9,12)\\9(11,21)10(11,14)11(18)12(13)13(15,17)14(15)16(17)18(19,20)20(21)$\\

$T(5,8):0(12,14,16)1(6,20,22)2(4,19,23)3(7,15,16)4(5,10)5(17,20)6(7,8)7(19)\\8(9,12)9(11,23)10(14,21)11(18,21)12(13)13(15,17)14(15)16(17)18(19,22)20(21)22(23)$\\


$T(6,7):0(12,14,16)1(6,20,22)2(4,19,23)3(7,15,16)4(5,10)5(17,20)6(7,8)7(19)\\8(9,12)9(11,23)10(14,21)11(18,21)12(13)13(15,17)14(15)16(17)18(19,22)20(21)22(23)$\\

$T(6,8):0(3,10,22)1(5,13,16)2(4,7,9)3(5,7)4(5,10)6(7,8,24)8(9,12)9(15)10(11)\\   11(13,20)12(17,25)13(25)14(19,21,22)15(18,21)16(17,24)17(19)18(19,23)20(21,23)22(23)\\24(25)$\\

$T(7,7):0(12,14,16)1(6,20,22)2(4,19,23)3(7,15,16)4(5,10)5(20,24)6(7,8)7(19)\\8(9,12)9(11,23)10(14,25)11(18,21)12(13)13(15,17)14(15)16(17)17(24)18(19,22)20(21)\\21(25)22(23)24(25)$\\


$T(8,8):0(8,10,14)1(5,9,11)2(7,16,18)3(13,19,22)4(5,7,28)5(13)6(7,11,29)8(9,29)\\9(12)10(11,28)12(13,17)14(15,23)15(17,18)16(17,20)18(21)19(23,26)20(21,27)21(24)\\22(25,27)23(25)24(25,26)26(27)28(29)$\\

$T(13):0(12,14,16)1(6,20,22)2(4,19,23)3(7,15,16)4(5,10)5(17,20)6(7,8)7(19)8(9,12)\\9(11,23)10(14,21)11(18,21)12(13)13(15,17)14(15)16(17)18(19,22)20(21)22(23)$\\



{\bf A2.} 
$T(8,8,8):=[0,3,4,7,18,17,22,21]\,\,[1,2,15,12,11,8,5,6]\,\,[9,10,23,20,19,16,13,14]$.

\begin{figure}[htpb] 
\centering\epsfig{file=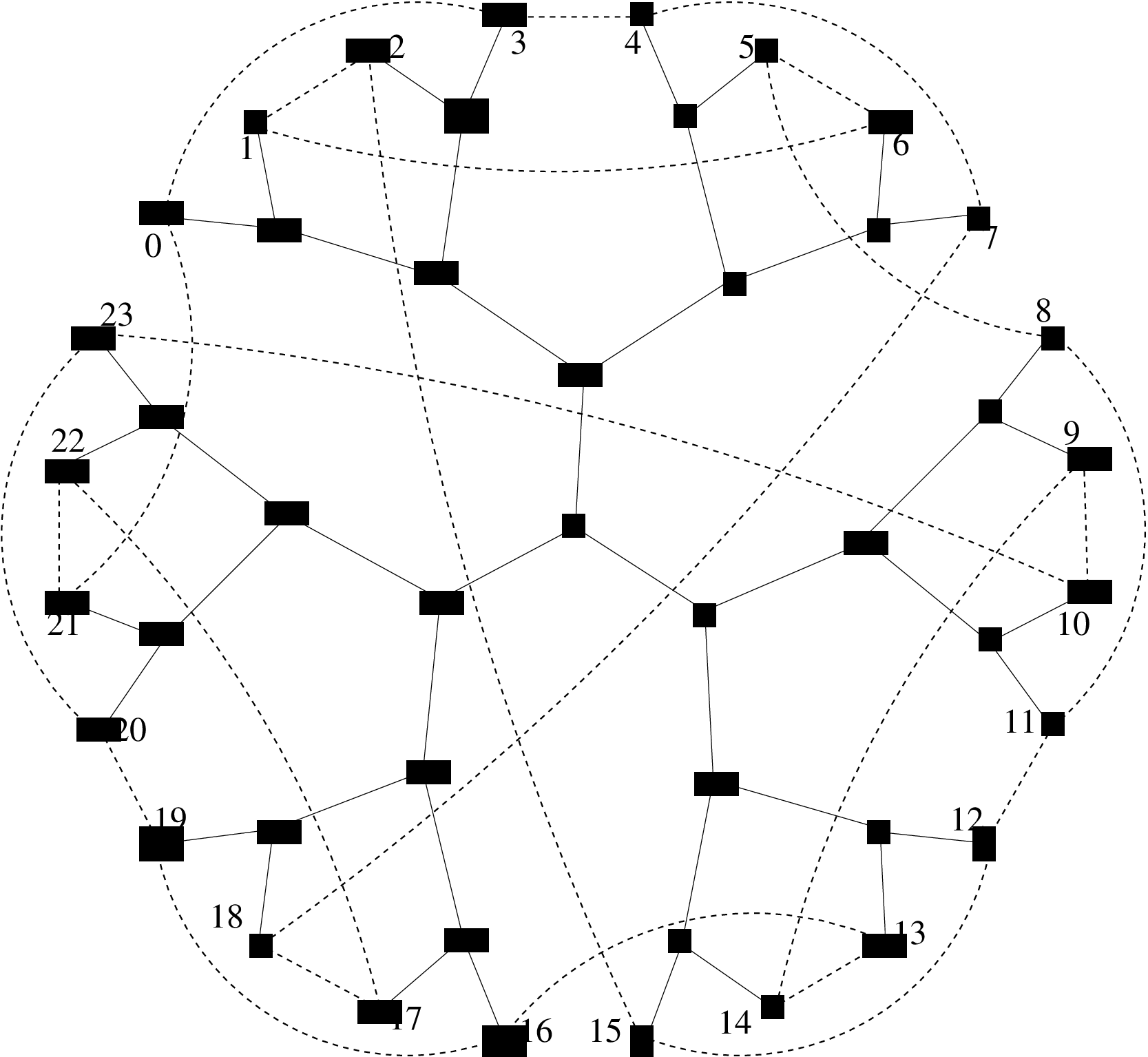,width=4.2in}
\caption{The hist-snark $T(8,8,8)$ with three outer cycles of length $8$.}
\label{f:t888}
\end{figure}

\smallskip

For the sake of completeness we present the adjacency list of $T(8,8,8)$.\\

\begin{small}
$T(8,8,8):0(3,21,24)1(2,6,24)2(15,25)3(4,25)4(7,26)5(6,8,26)6(27)7(18,27)8(11,28)\\9(10,14,28)10(23,29)11(12,29)12(15,30)13(14,16,30)14(31)15(31)16(19,32)17(18,22,32)\\18(33)19(20,33)20(23,34)21(22,34)22(35)23(35)24(36)25(36)26(37)27(37)28(38)29(38)30(39)\\31(39)32(40)33(40)34(41)35(41)36(42)37(42)38(43)39(43)40(44)41(44)42(45)43(45)44(45)$
\end{small}
\\


{\bf A3.} The adjacency lists of the hist-free snarks $X_1$, $X_2$ with $38$ vertices.\\

\begin{small}

$X_1:0(8,12,18)1(5,9,13)2(4,14,20)3(5,7,8)4(5,12)6(7,10,13)7(14)8(15)9(19,22)\\10(18,24)11(26,34,36)12(16)13(16)14(17)15(17,19)16(17)18(21)19(21)20(25,36)21(27)\\22(30,34)23(25,28,31)24(26,37)25(35)26(32)27(29,31)28(29,30)29(32)30(33)31(33)32(33)\\34(35)35(37)36(37)$\\

$X_2:0(8,12,18)1(5,9,13)2(4,14,20)3(5,7,8)4(5,12)6(7,10,13)7(14)8(15)9(19,22)\\10(18,24)11(26,34,36)12(16)13(16)14(17)15(17,19)16(17)18(21)19(21)20(28,34)21(27)\\22(26,37)23(27,30,32)24(25,36)25(30,35)26(33)27(29)28(31,32)29(31,33)30(31)32(33)\\34(35)35(37)36(37)$
\\
\end{small}




\begin{thebibliography}{1}

\bibitem {AT} G.M.Adelson-Velskij and V.K.Titov, On 4-chromatic cubic graphs, (Proc. Seminar of 1971 at Moscow Univ.).
\newblock Voprosy Kibernetiki 5-14 (1973), in Russian.


\bibitem {ABHT} M.O.Albertson, D.M.Berman, J.P.Hutchinson, C.Thomassen, Graphs with homeomorphically irreducible spanning trees.
\newblock Journal of Graph Theory 14 (2) (1990) 247-258.

\bibitem {A} L.D.Andersen, H.Fleischner, B.Jackson, Removable edges in cyclically 4-edge-connected cubic graphs.
\newblock Graphs and Combinatorics 4 (1988) 1-21. 

\bibitem {BM} J.A.Bondy, U.S.R.Murty,
\newblock Graph Theory, Springer (2008).


\bibitem {BGHM} G.Brinkmann, J.Goedgebeur, J.H\"agglund, K.Markstr\"om, Generation and properties of snarks.
\newblock J. Combin. Theory Ser. B 103 (2013) 468-488. 


\bibitem {BCGM} G.Brinkmann, K.Coolsaet, J.Goedgebeur, H.Melot, House of Graphs: a database of interesting graphs. \newblock Discrete Applied Mathematics 161 (2013) 311-314. Available at http://hog.grinvin.org 

\bibitem {F1} M.Fontet, Graphes 4-essentiels. C.R. Acad. Sci., Paris, Ser. A287 (1978) 289-290.

\bibitem {F2} M.Fontet, Connectivit\'{e} des graphes automorphismes des cartes: propri\'{e}t\'{e}s et algorithmes.  
\newblock Paris; Th\'{e}se d'Etat, Universit\'{e} P. et M. Curie 1979.


\bibitem {HO3} A.Hoffmann-Ostenhof,
\newblock unpublished manuscript (2017).

\bibitem {HJ} A.Hoffmann-Ostenhof, T.Jatschka,
\newblock Special Hist-Snarks. arXiv:1710.05663.

\bibitem {HKO} A.Hoffmann-Ostenhof, T.Kaiser, K.Ozeki,\\
\newblock Decomposing planar cubic graphs. Journal of Graph Theory 88 (4) (2018) 631-640.


\bibitem {HO} A.Hoffmann-Ostenhof, K. Noguchi, K.Ozeki,\\
\newblock On Homeomorphically Irreducible Spanning Trees in Cubic Graphs. \\
\newblock Journal of Graph Theory 89 (2) (2018) 93-100.



\bibitem {HZZ} A.Hoffmann-Ostenhof, C.Q.Zhang, Z.Zhang,\\
\newblock Cycle double covers and non-separating cycles. arXiv:1711.10614v2. 








\bibitem {I} R.Isaacs, Infinite families of non-trivial trivalent graphs which are not Tait-colorable.\\ American Mathematical Monthly 82 (3) (1975) 221-239. 



\bibitem {W} N.C.Wormald: Classifying k-connected cubic graphs, Lecture Notes in Mathematics vol. 748, 199-206. 
New York: Springer-Verlag (1979).



\bibitem {Z1} C.Q.Zhang, Integer flows and cycle covers of graphs. CRC Press (1997).















\end{thebibliography}
\end{document}